\documentclass[11pt, a4paper, reqno]{amsart}

\newcommand{\ga}{\alpha}
\newcommand{\gb}{\beta}
\newcommand{\gd}{\delta}

\newcommand{\go}{\omega}

\newcommand{\vp}{\varphi}

\newcommand{\gO}{\Omega}

\newcommand{\cA}{\mathcal{A}}

\newcommand{\N}{\mathbb{N}}

\newcommand{\R}{\mathbb{R}}

\newcommand{\Exp}{\ensuremath {{\rm Exp}}}
\let\comp\circ

\newtheorem{theorem}{Theorem}[section]

\newtheorem{lemma}[theorem]{Lemma}
\newtheorem{corollary}[theorem]{Corollary}

\theoremstyle{definition}

\theoremstyle{remark}

\newtheorem{remark}[theorem]{Remark}

\newtheoremstyle{break}{9pt}{9pt}{\upshape}{}{\bfseries}{.}{\newline}{}
\theoremstyle{break}

\newtheorem{Assumptions}[theorem]{Assumptions}

\newenvironment{enum_a}
    {\begin{enumerate}\renewcommand{\labelenumi}{(\alph{enumi})}}
    {\end{enumerate}}
\newenvironment{enum_i}
    {\begin{enumerate}\renewcommand{\labelenumi}{(\roman{enumi})}
		}
    {\end{enumerate}}

\allowdisplaybreaks[1]
\numberwithin{equation}{section}

\title{Continuity of Random Fields on Riemannian Manifolds}
\dedicatory{Dedicated to the memory of Robert Schrader}

\author[A.~Lang]{Annika Lang}
\address{Annika Lang\newline
Department of Mathematical Sciences\newline
Chalmers University of Technology \& University of Gothenburg\newline
SE--41296 G{\"o}teborg, Sweden}
\email{annika.lang@chalmers.se}

\author[J.~Potthoff]{J\"urgen Potthoff}
\address{J\"urgen Potthoff\newline
Institut f\"ur Mathematik, Universit\"at Mannheim\newline
D--68131 Mannheim, Germany}
\email{potthoff@uni-mannheim.de}

\author[M.~Schlather]{Martin Schlather}
\address{Martin Schlather\newline
Institut f\"ur Mathematik, Universit\"at Mannheim\newline
D--68131 Mannheim, Germany}
\email{schlather@uni-mannheim.de}

\author[D.~Schwab]{Dimitri Schwab}
\address{Dimitri Schwab\newline
Institut f\"ur Mathematik, Universit\"at Mannheim\newline
D--68131 Mannheim, Germany}
\email{dschwab@mail.uni-mannheim.de}

\date{July 20, 2016}



\begin{document}

\begin{abstract} 
A theorem of the Kolmogorov--Chentsov type is proved
for random fields on a Riemannian manifold.
\end{abstract}
\maketitle
\section{Introduction} \label{sectIntro}
One of the key theorems in the theory of stochastic processes is the 
Kolmogorov--Chentsov theorem (the classical references are~\cite{Sl37} and~\cite{Ch56}), 
which states the existence of a continuous modification of a given stochastic process 
based on tail or moment estimates of its increments.

In the present paper we prove a theorem of the Kolmogorov--Chentsov type for
random fields indexed by a finite--dimensional Riemannian manifold. A result of similar kind has 
been proved  in a recent paper, in which also results about differentiablity are shown~\cite{AnLa14}.
Concerning the Kolmogorov--Chentsov theorem the differences between the present
and the cited paper are twofold: For one, the continuity statement in~\cite{AnLa14}
is formulated in terms of coordinate mappings, while here we give an intrinsic
formulation, i.e., a direct formulation in terms of the Riemannian (topological) metric.
Secondly, the methods of proof are quite different: While the Sobolev embedding 
theorem is employed in~\cite{AnLa14}, we basically use here the classical
method via dyadic approximations and the Borel--Cantelli lemma. In fact,
our proof of the Kolmogorov--Chentsov theorem for random fields on a Riemannian
manifold is based on a localized variant of a theorem in~\cite{Po09a}, which is 
combined with a  local coordinatization of the underlying Riemannian manifold in 
terms of the  exponential map.

The organization of the paper is as follows. In Section~\ref{sectMain}
we recall the setup of the general Kolmogorov--Chentsov theorem in~\cite{Po09a}, and 
prove the above mentioned localized version of Theorem~2.8 in~\cite{Po09a}. In 
Theorem~\ref{thm2} below we state and prove our main assertion,
namely our theorem of Kolmogorov--Chentsov type for Riemannian manifolds. Finally, in 
Section~\ref{sectHolder} we discuss the existence of locally H\"older continuous
modifications, provide sufficient conditions in terms of moments of the increments,
and consider the special case of Gaussian random fields as examples.

\vspace*{.5\baselineskip}\noindent
\textbf{Acknowledgement.} We are grateful to Martin Schmidt for fruitful
discussions. The work of AL was supported in part by the Swedish Research Council 
under Reg.~No.~621-2014-3995 as well as the Knut and Alice~Wallenberg foundation.

\section{A Kolmogorov--Chentsov Theorem for Metric Spaces} \label{sectMain}
In this section we give a variant of the Kolmogorov--Chentsov type
theorem in~\cite{Po09a}, which follows rather directly from it, and 
in some sense sharpens that result.

Suppose that $(M,d)$ is a separable metric space, that $(\gO,\cA,P)$ is a  
probability space, and that $\phi = \bigl(\phi(x),\,x\in M\bigr)$ 
is a real-valued  random field on this probability space indexed by $M$. 

Assume furthermore that $r$ and $q$ are two strictly increasing functions
on an interval $[0,\rho)$, $\rho>0$, such that $r(0)=q(0)=0$. Throughout
this paper we suppose that for all $x$, $y\in M$ with $d(x,y)<\rho$,
we have the bound
\begin{equation}	\label{phibound}
	P\Bigl(\bigl|\phi(x)-\phi(y)\bigr|> r\bigl(d(x,y)\bigr)\Bigr) 
			\le q\bigl(d(x,y)\bigr).
\end{equation}

\begin{remark}
We take the occasion to correct a minor error in~\cite{Po09a}: There, this bound
has been formulated and used with a ``$\ge$'' sign for the event under the
probability, which for $x=y$ is obviously absurd. However, an inspection shows that
all arguments and results in~\cite{Po09a} remain correct, when replacing the 
condition there with the inequality~\eqref{phibound}. Alternatively, the condition 
in~\cite{Po09a} could be supplied with the additional restriction~$x\ne y$.
\end{remark}

We make the following assumptions on the metric space $(M,d)$:
\begin{Assumptions}\mbox{} \label{Assmpt}
\vspace*{-1\baselineskip}
\begin{enum_a}
	\item There exists an at most countable open cover $(U_n,\,n\in\N)$
			of $M$, and for every $n\in\N$, there exists a metric $d_n$ on $U_n$
			so that $\ga_n\, d_n(x,y)\le d(x,y)\le d_n(x,y)$ for all $x$, $y\in U_n$
			and some $\ga_n\in (0,1]$;
	\item for every $n\in\N$, $(U_n,d_n)$ is well separable in the sense of~\cite{Po09a}, i.e.:
			\begin{enum_i}
				\item there exists an increasing sequence $(D_{n,k},\,k\in\N)$
					  of finite subsets of $U_n$ such that $D_n = \bigcup_k D_{n,k}$
					  is dense in $(U_n,d_n)$, 
					  and for $x\in D_{n,k}$, let $C_{n,k}(x)= \{y\in D_{n,k},\,d_n(x,y)
					  \le \gd_{n,k}\}$, where $\gd_{n,k}$ denotes the minimal distance 
					  of distinct points in $D_{n,k}$ with respect to $d_n$;
				\item every $z\in U_n$ has a neighborhood $V\subset U_n$ so that for almost
                	  all $k\in\N$ and all $x$, $y\in D_{n,k+1}\cap V$, there exist $x'$, 
					  $y'\in D_{n,k}\cap V$ with $x'\in C_{n,k+1}(x)$,
					  $y'\in C_{n,k+1}(y)$, and $d_n(x',y')\le d_n(x,y)$;
			\end{enum_i}
				\item for $n$, $k\in\N$, let $\pi_{n,k}$ be the set of all unordered pairs 
					  $\langle x,y\rangle$, $x$, $y\in D_{n,k}$ with $d_n(x,y)\le \gd_{n,k}$, 
					  and let $|\pi_{n,k}|$ denote the number of elements in this set, then
					  \begin{align}
					   \sum_k |\pi_{n,k}|\,q(\gd_{n,k}) &<+\infty, \label{ineqpiq}\\
					   \sum_k r(\gd_{n,k}) &<+\infty \label{ineqr}
				      \end{align}
					  hold true.
\end{enum_a}
\end{Assumptions}

We remark that due to the assumption on the metrics $d$ and $d_n$, $n\in\N$, in~(a) above, 
the relative topology on $U_n$ generated by $d$ coincides with the topology generated by $d_n$.

For $x$, $y\in U_n$ with $d(x,y)<\ga_n \rho$, we can estimate as follows
\begin{align*}
	P\Bigl(\bigl|\phi(x)-\phi(y)\bigr|&> r\bigl(d_n(x,y)\bigr)\Bigr)\\
		&\le P\Bigl(\bigl|\phi(x)-\phi(y)\bigr|> r\bigl(d(x,y)\bigr)\Bigr)\\
		&\le q\bigl(d(x,y)\bigr)\\
		&\le q\bigl(d_n(x,y)\bigr),
\end{align*}
because $r$ and $q$ are both increasing. Theorem~2.8 in~\cite{Po09a} shows that from this 
estimate, together with the Assumptions~(a), (b), and~(c) above, it follows that for every
$n\in\N$, the restriction $\phi_n$ of $\phi$ to $U_n$ has a locally uniformly continuous 
modification $\psi_n$ which is such that $\phi_n$, $\psi_n$, and $\phi$ coincide on $D_n$. 
In more detail we have that for every $n\in\N$, there exists a random field $\psi_n$ indexed by $U_n$
such that
\begin{enum_i}
	\item for every $\go\in\gO$, the mapping $\psi_n(\cdot,\go): U_n \to \R$
			is locally uniformly continuous;
	\item for every $x\in U_n$, there exists a $P$--null set $N_{x,n}$ so that
		 	$\psi_n(x,\go) = \phi(x,\go)$ for all $\go$ in the complement of $N_{x,n}$,
			and if $x\in D_n$, $N_{x,n}$ can be chosen as the empty set.
\end{enum_i}
In order to get for $x\in M$ a universal $P$-null set $N_x$, we set
$N_x=\bigcup_{n'} N_{x,n'}$, where the union is over all $n'\in\N$ such
that $x\in U_{n'}$. Since this is a countable union, $N_x$ is indeed a $P$--null set.

From the modifications $\psi_n$ of $\phi_n$, $n\in\N$, we construct a 
locally uniformly continuous modification $\psi$ of $\phi$. We show

\begin{lemma}	\label{lem1}
\begin{equation*}
	P\bigl(\psi_n(x)=\psi_{n'}(x),\,x\in U_n\cap U_{n'},\,n,n'\in\N\bigr)=1.
\end{equation*}
\end{lemma}

\begin{proof}
Assume that $x\in U_n\cap U_{n'}$. Since $\psi_n$ and $\psi_{n'}$ are modifications
of $\phi$ when all these random fields are restricted to $U_n\cap U_{n'}$, we get
$\psi_n(x)=\psi_{n'}(x)$ on the complement of the $P$--null set $N_x$. Since
$(M,d)$ is separable so is $(U_n\cap U_{n'},d)$, and letting $x$ range over a countable
dense subset $E_{n,n'}$ and taking the union of all associated $P$--null sets, we  get the
existence of a $P$--null set $N_{n,n'}$ such that for all $x\in E_{n,n'}$, we have
$\psi_n(x) = \psi_{n'}(x)$ on the complement of $N_{n,n'}$. $\psi_n$ and $\psi_{n'}$
are continuous on $U_n\cap U_{n'}$, and hence we obtain for all $x\in U_n\cap U_{n'}$
the equality $\psi_n(x)=\psi_{n'}(x)$ on the complement of $N_{n,n'}$. Finally, we
set $N = \bigcup_{n,n'} N_{n,n'}$ so that we find for all $n$, $n'$, and all
$x\in U_n\cap U_{n'}$ the equality $\psi_n(x)=\psi_{n'}(x)$ on the complement of the
$P$--null set $N$.
\end{proof}

On the exceptional set~$N$ of the last lemma we define $\psi(x)=0$ for all $x\in M$.
On its complement we set $\psi(x) = \psi_n(x)$ whenever $x\in U_n$, and the last
lemma shows that this makes $\psi$ well-defined. For $x\in M$, define the $P$--null
set $N'_x = N_x\cup N$ where $N_x$ and $N$ are the $P$--null sets defined above.
We have $x\in U_n$ for some $n\in\N$, and for all $\go$ in the complement of $N'_x$, 
we find that $\psi(x,\go) = \psi_n(x,\go) = \phi(x,\go)$. Thus $\psi$ is a
modification of $\phi$. We have proved the following

\begin{theorem}	\label{thm1}
Under Condition~\eqref{phibound} on the random field $\phi$ and under the above
Assumptions~\ref{Assmpt} on $(M,d)$, $r$ and $q$, $\phi$ has a locally uniformly 
continuous modification.
\end{theorem}

\section{A Kolmogorov--Chentsov Theorem for Riemannian Manifolds} \label{sectPf}
We begin this section by recalling the necessary terminology of Riemannian geometry,
setting up our notation at the same time. For further background
the interested reader is referred to the standard literature, e.g.\ \cite{He78, 
Jo11, Pe06}.

Assume that $m\in\N$ and that~$(M,g)$ is an~$m$--dimensional Riemannian manifold as 
defined in~\cite{He78}. That is, $M$ is a connected, $m$--dimensional 
$C^\infty$--manifold together with a symmetric, strictly positive definite tensor 
field~$g$ of type~$(0,2)$. For each $x\in M$, the Riemannian metric~$g$ determines
an inner product $g_x(\cdot,\cdot)$ on the tangent space $T_x M$ at $x$:
\begin{align*}
	g_x: T_x M\times T_x M &\to \R\\
		 (X,Y) &\mapsto g_x(X,Y).
\end{align*}
The corresponding norm on $T_x M$ is given by 
\begin{equation*}
	\|X\| = g_x(X,X)^{1/2},\qquad X\in T_x M.
\end{equation*}

Let $c: [a,b] \to M$ be a smooth curve in~$M$. Then its derivative $c'(t)$ at $t\in (a,b)$
belongs to $T_{c(t)}M$, and the length of $c$ is given by
\begin{equation*}
	L(c)=\int_a^b \|c'(t)\| \, dt.	
\end{equation*}
The Riemannian distance $d(x,y)$ of two points $x$, $y\in M$ is defined as the infimum of 
the lengths of curve segments joining $x$ and $y$. Indeed, $d$ is a metric on $M$ and it can be 
shown that under the given assumptions on $M$ the metric space $(M,d)$ 
is separable, locally compact and connected \cite[Proposition~I.9.6]{He78}. 
Furthermore, the original topology and the topology defined by~$d$ 
coincide~\cite[Corollary~I.9.5]{He78}.

We denote the open ball of radius $R>0$ centered at $x\in M$ relative to the metric $d$ by 
$B_R^d(x)$, while the ball of radius $R$ in $T_xM$ with center at $X\in T_xM$ with respect to 
the norm $\|\cdot\|$ is denoted by $B_R(X)$.

With the Riemannian metric $g$ there is canonically associated --- via
the notions of parallel transport and geodesics ---
the exponential map $(\Exp_x,\,x\in M)$, which for each $x\in M$ is a mapping
from $T_x M$ into a neighborhood of~$x$ in~$M$. It can be shown 
that for each $x\in M$, there exists a radius $R(x)>0$ such that $\Exp_x$ maps 
$B_{R(x)}(0)$ diffeomorphically onto $B^d_{R(x)}(x)$~\cite[Theorem~I.9.9, Proposition~I.9.4]{He78}. 
Moreover, for all $Y$, $Z\in B_{R(x)}(0)$ such that $\Exp_x(Y)=y$, $\Exp_x(Z)=z$, the 
quotient $\| Y-Z\|/d(y,z)$ converges to $1$ as 
$(y,z)\rightarrow (x,x)$ \cite[Proposition~I.9.10]{He78}. 

In view of Theorem~\ref{thm1} in Section~\ref{sectMain}, we construct a countable cover 
$(U_n,\,n\in\N)$ of $M$ as follows. The separability of $M$ (see above) allows us
to fix a countable dense subset $\{x_n,\,n\in\N\}$ of $M$. For every $n\in\N$, choose 
$R_n\in \bigl(0,1/(2\sqrt{m})\bigr]$ in such a way that:
\begin{enumerate} \renewcommand{\labelenumi}{\arabic{enumi}.}
	\item the exponential map $\Exp_{x_n}$ is a diffeomorphism from $B_{R_n}(0)\subset T_{x_n}M$ onto 
			$B^d_{R_n}(x_n)\subset M$,
	\item for all $X$, $Y\in B_{R_n}(0)$ such that $\Exp_{x_n}(X)=x$, $\Exp_{x_n}(Y)=y$,
			$x$, $y\in B^d_{R_n}(x_n)$,
			\begin{equation} \label{metrineq}
				 2^{-1}\|X-Y\| \le d(x,y) \le 2 \|X-Y\|.
			\end{equation}
\end{enumerate}
The existence of a strictly positive $R_n$ for each $n\in\N$ with these properties follows 
from the facts mentioned before.

The idea is now to use the exponential map in order to define a convenient
coordinatization of $U_n$ and to use inequality~\eqref{metrineq} for
the definition of a suitable metric $d_n$ on $U_n$. To this end, we fix
an orthonormal basis $(X_{n,1},\dotsc, X_{n,m})$ of $(T_{x_n} M,g_{x_n})$ so that
every $X\in T_{x_n}M$ can be written in a unique way as
\begin{equation*}
	X = \sum_{i=1}^m a_i X_{n,i}
\end{equation*}
with $a=(a_1,\dotsc, a_m)\in \R^m$. Let us denote the so defined linear
mapping from $\R^m$ onto $T_{x_n}M$ by $L_{x_n}$. The orthonormality
of $(X_{n,1},\dotsc, X_{n,m})$ entails that $L_{x_n}$ is 
an isometric isomorphism if $\R^m$ is equipped with the standard euclidean
metric. In particular, the ball $B_{R_n}(0)$ is under $L_{x_n}$ in one-to-one
correspondence with the euclidean ball $B^m_{R_n}(0)$ in $\R^m$. Define
\begin{equation*}
	\vp_n(x) = L_{x_n}^{-1}\comp \Exp_{x_n}^{-1}(x),\qquad x\in B^d_{R_n}(x_n),
\end{equation*}
then $\vp_n$ is a $C^\infty$--coordinatization of $B^d_{R_n}(x_n)$ which 
maps this ball onto $B_{R_n}^m(0)\subset \R^m$. 

For $x$, $y\in B^d_{R_n}(x_n)$, define
\begin{equation}	\label{maxmetr}
	d_n(x,y) = 2\sqrt{m}\,\max_{i=1,\dotsc,m}\bigl|\vp_n^i(x)-\vp_n^i(y)\bigr|,
\end{equation}
where $\vp_n^i(x)$ denotes the $i$--th Cartesian coordinate of $\vp_n(x)$.
Set $\ga_n = 1/(4\sqrt{m})$. If $\|\,\cdot\,\|_2$ denotes the usual euclidean
norm on $\R^m$, we obtain from~\eqref{metrineq}
\begin{align*}
	\ga_n\, d_n(x,y)
		&= 2^{-1} \max_i \bigl|\vp_n^i(x)-\vp_n^i(y)\bigr|\\
		&\le 2^{-1}\|\vp_n(x)-\vp_n(y)\|_2\\
		&= 2^{-1} \|\Exp_{x_n}^{-1}(x)-\Exp_{x_n}^{-1}(y)\|\\[1ex]
		&\le d(x,y)\\[1ex]
		&\le 2 \|\Exp_{x_n}^{-1}(x)-\Exp_{x_n}^{-1}(y)\|\\[1ex]
		&= 2 \|\vp_n(x)-\vp_n(y)\|_2\\[1ex]
		&\le d_n(x,y).
\end{align*}

\vspace*{.5\baselineskip}
Consider the open hypercube $H_{R_n}^m(0)$
\begin{equation*}
	H_{R_n}^m(0) = \bigl\{x\in\R^m,\,\max_{i=1,\dotsc, m}|x_i|< m^{-1/2} R_n\bigr\} 
\end{equation*}
in $\R^m$ of sidelength $2m^{-1/2}R_n$ centered at the origin. Clearly we have
$H_{R_n}^m(0)\subset B_{R_n}^m(0)$. Set
\begin{equation*}
	U_n = \vp_n^{-1}\bigl(H_{R_n}^m(0)\bigr)
\end{equation*}
so that $(U_n,\,n\in\N)$ is an open cover of~$M$. 

For each $k\in\N$, define the following subset $G_{n,k}$ of the hypercube $H_{R_n}^m(0)$
\begin{equation*}
	G_{n,k} = \Bigl\{a\in\R^m,\, a = -\frac{R_n}{\sqrt{m}} + \frac{lR_n}{2^k \sqrt{m}},\,
					l\in \bigl\{1,\dotsc, 2^{k+1}-1\bigr\}^m\Bigr\}.
\end{equation*}
By construction, for each $n \in \N$, $(G_{n,k},\,k\in\N)$ is an increasing sequence of 
finite subsets of $H^m_{R_n}(0)$, and the union of these sets is dense in $H^m_{R_n}(0)$. 
Next set $D_{n,k} = \vp_n^{-1}(G_{n,k})$. Then for each $n\in\N$, $(D_{n,k},\,k\in\N)$
is an increasing sequence of subsets of $U_n$, its limit being dense in $U_n$.
Moreover, it is easy to see that Condition~(b.ii) of Assumption~\ref{Assmpt} holds
true for the sequence $(D_{n,k},\,k\in\N)$, where for every $z\in U_n$, we may
choose the neighborhood $V$ in this condition as $U_n$ itself. (For an
explicit argument, see also~\cite{Po09a}.)

By construction we have (in terms of the notation of Section~\ref{sectMain})
\begin{equation*}
	\gd_{n,k} = \min\,\bigl\{d_n(x,y),\,x,y\in D_{n,k},\,x\ne y\bigr\} 
		= 2^{-k+1} R_n. 
\end{equation*}
The number $|\pi_{n,k}|$ of unordered pairs in $\pi_{n,k}$ (cf.\ Assumption~\ref{Assmpt}.(c)) 
can be bounded from above by $K_m 2^{mk}$ for some constant $K_m$. 

Now let $\rho\in(0,1]$, and make the usual choices of the functions $r$, and $q$:
\begin{align}
	r(h) &= \log_2(h^{-1})^{-\gb}, \label{defr}\\
	q(h) &= K\,\log_2(h^{-1})^{-\ga}\,h^m \label{defq},
\end{align}
for $h\in(0,\rho)$, and $r(0)=q(0)=0$. Here $K > 0$ is an arbitrary constant, and $\ga$, $\gb>1$.
Then it is straightforward to check that the inequalities~\eqref{ineqpiq}, \eqref{ineqr}
hold true.

Thus we can apply Theorem~\ref{thm1} and obtain

\begin{theorem}	\label{thm2}
Suppose that $\phi$ is a random field defined on an $m$--dimensional Riemannian 
manifold $(M,g)$ with topological metric $d$ such that for all $x$, $y\in M$ with $d(x,y)<\rho$,
\begin{equation*}
	P\Bigl(\bigl|\phi(x)-\phi(y)\bigr|> r\bigl(d(x,y)\bigr)\Bigr) 
			\le q\bigl(d(x,y)\bigr)
\end{equation*}
holds true, where the functions $r$, $q$ are defined as in~\eqref{defr}, \eqref{defq} for
some constants $K>0$, $\ga$, $\gb>1$. Then $\phi$ has a locally uniformly continuous 
modification.
\end{theorem}

\section{H\"older Continuity and Moment Conditions}\label{sectHolder}
While a locally uniformly continuous modification was constructed in the 
previous section, the goal here is to show higher regularity in terms of 
orders of H\"older continuity under additional assumptions.

Therefore, let $(M,d)$ be a metric space and $\phi$ be as in Section~\ref{sectMain}, 
Assumption~\ref{Assmpt}. For the sequences of minimal distances $(\delta_{n,k},k\in\N)$ of 
distinct points in the grids $(D_{n,k},k\in\N)$ and the function $r$, we make the stronger 
assumptions

\goodbreak
\begin{Assumptions}{\ }	\label{Assmptii}
\vspace*{-1\baselineskip}
\begin{enum_a}
	\item For every $n\in\N$, there exist constants $\eta_n\in(0,1)$, $C_n>0$ such that 
				for almost all $k\in\N$, 
				\begin{equation}	\label{eqHolderi}
					\frac{1}{C_n} \, \eta_n^{k}\leq \delta_{n,k}\leq C_n\, \eta_n^k
				\end{equation}
				holds true;
	\item  there exist constants $\gamma\in (0,1)$, $K_\gamma>0$ so that
				for all $h\in [0,\rho)$, the inequality 
				\begin{equation}\label{ineqrnew}
					r(h)\le K_\gamma h^\gamma
				\end{equation}
				is valid.
\end{enum_a}
\end{Assumptions}

\noindent
(Note that Assumption~(b) on~$r$ above is stronger than the requirement of inequality~\eqref{ineqr}.)
Then, on every $(U_n, d_n)$, we are in the situation of  Theorem~2.9 in \cite{Po09a} and 
get the existence of a modification $\psi_n$ which is locally H\"older continuous of 
order $\gamma$, i.e., for every $\omega\in\Omega$ and every $z\in U_n$, there exists a 
neighborhood $V(\omega)$ of $z$ in $(U_n,d_n)$ and a constant $\ga_{\gamma,n}$ such that
\begin{equation*}
	\sup_{x,y\in V(\omega),\, x\neq y} 
		\left| \frac{\psi_n(x,\omega)-\psi_n(y,\omega)}{d_n(x,y)^{\gamma}}\right |\leq \ga_{\gamma,n}.
\end{equation*}
Actually, the constant $\ga_{\gamma,n}$ was explicitly calculated in~\cite{Po09a} and is
given by
\begin{equation*}
	\ga_{\gamma, n} = 2 K_\gamma\,\frac{C_n^{2\gamma}}{\eta_n^\gamma (1-\eta_n^\gamma)}.
\end{equation*}

Again we can glue these modifications together to get a modification $\psi$ of $\phi$ on $(M,d)$ 
which is locally H\"older continuous of order~$\gamma$. 

\begin{corollary}	\label{cori}
Assume that Condition~\eqref{phibound} on the random field $\phi$ holds true. Suppose
furthermore that the Assumptions~\ref{Assmpt} are valid, together with the additional stronger
properties given in Assumptions~\ref{Assmptii}. Then $\phi$ has a modification which is 
locally H\"older continuous of order~$\gamma$.
\end{corollary}

We return to the case where $M$ is an $m$--dimensional Riemannian manifold with topological 
metric~$d$. Let the open cover $((U_n,d_n),n\in\N)$, and the sequences of grids 
$((D_{n,k},\delta_{n,k}),k\in\N)$, $n\in\N$, be defined as in Section~\ref{sectPf}. Recall 
that $\delta_{n,k}=2^{-k+1}R_n$ and $R_n\in(0,1/2\sqrt{m})]$, $n$, $k\in\N$. Set $\eta_n=1/2$, 
and choose $C_n\geq 1/(2R_n)$. Then Condition~\eqref{eqHolderi} is fulfilled. As before let 
$\rho$ be in $(0,1]$. Define $q$ as in \eqref{defq}, and 
\begin{equation}\label{defrnew}
	r(h)=h^{\gamma}
\end{equation}
for some $\gamma\in(0,1)$. Then Condition~\eqref{ineqrnew} is valid as well, and so we arrive at

\begin{corollary} \label{corii}
Let $\phi$ be a random field defined on an $m$--dimensional Riemannian manifold $(M,g)$, 
$m\in\N$, with topological metric $d$, such that for all $x$, $y\in M$ with $d(x,y)<\rho$,
\begin{equation*}
	P\Bigl(\bigl|\phi(x)-\phi(y)\bigr|> r\bigl(d(x,y)\bigr)\Bigr) 
			\le q\bigl(d(x,y)\bigr)
\end{equation*}
holds true, where the functions $r$, $q$ are defined as in \eqref{defrnew}, \eqref{defq} for
some constants $K>0$, $\ga>1$, $\gamma\in(0,1)$. Then $\phi$ has a locally H\"older continuous 
modification of order~$\gamma$.
\end{corollary}

The standard application of Chebychev's inequality yields sufficient conditions in terms of moments: 

\begin{corollary}	\label{coriii}
Suppose that $\phi$ is a random field defined on an $m$--dimensional Riemannian manifold $M$, $m\in\N$, 
with topological metric $d$.
\begin{enum_a}
	\item If there exist $\rho\in(0,1]\,$, $l\geq 1$, $\kappa\geq m$, $\nu>l+1$, and $K>0$ such that 
			\begin{equation*}
				\mathbb{E}\bigl(|\phi(x)-\phi(y)|^l\bigr)
					\leq K\log_2\bigl(d(x,y)^{-1}\bigr)^{-\nu}d(x,y)^{\kappa}
			\end{equation*}
			for all $x$, $y\in M$ with $d(x,y)< \rho$, then $\phi$ has a modification which 
			is locally uniformly continuous.
	\item If there are $\rho\in(0,1]\,$, $l\geq 1$, $\gamma\in(0,1)$, and $\alpha>1$ such that
			\begin{equation*}
				\mathbb{E}\bigl(|\phi(x)-\phi(y)|^l\bigr)
					\leq K\log_2\bigl(d(x,y)^{-1}\bigr)^{-\ga}d(x,y)^{m+l\gamma}
			\end{equation*}
			for all $x$, $y\in M$ with $d(x,y)<\rho$, the modification can be chosen to have 
			locally H\"older continuous sample paths of order~$\gamma$. 
\end{enum_a}
\end{corollary}

In case of a Gaussian random field, Corollary~\ref{coriii} leads to a condition which can be formulated 
in terms of the variogram of the random field:

\begin{corollary}	\label{coriv}
Assume that $\phi$ is a Gaussian random field defined on an $m$--dimensional 
Riemannian manifold $M$, $m\in\N$, with topological metric $d$ and variogram 
$\sigma(x,y)^2=\mathbb{E}\bigl((\phi(x)-\phi(y))^2\bigr)$. If there exist $\rho\in(0,1]\,$, 
$\eta \in (0,1)$, and $C>0$ such that
\begin{equation}	\label{cond_vario}
	\sigma(x,y)^2\leq C \,d(x,y)^{\eta}
\end{equation}
for all $x$, $y\in M$ with $d(x,y)< \rho$, 
then $\phi$ has a modification which is locally H\"older continuous of order $\gamma$ for all 
$\gamma<\eta/2$.
\end{corollary}

\vspace*{\baselineskip}
Applied to the specific case of the $m$--dimensional unit sphere embedded in~$\R^{m+1}$, 
this corollary recovers the results from~\cite{LaSch15}, where isotropic Gaussian random 
fields on spheres are considered.

\providecommand{\bysame}{\leavevmode\hbox to3em{\hrulefill}\thinspace}
\providecommand{\MR}{\relax\ifhmode\unskip\space\fi MR }
\providecommand{\MRhref}[2]{%
  \href{http://www.ams.org/mathscinet-getitem?mr=#1}{#2}
}
\providecommand{\href}[2]{#2}

\setlength{\parindent}{0pt}


\begin{thebibliography}{1}

\bibitem{AnLa14}
Andreev, R. and Lang, A.: {K}olmogorov-{C}hentsov theorem and differentiability
  of random fields on manifolds, \emph{Potential Analysis} \textbf{41} (2014)
  761--769.

\bibitem{Ch56}
Chentsov, N.~N.: {W}eak convergence of stochastic processes whose trajectories
  have no discontinuities of the second kind and the ``heuristic'' approach to
  the {K}olmogorov-{S}mirnov tests, \emph{Theory Probab. Appl.} \textbf{1}
  (1956) 140--144.

\bibitem{He78}
Helgason, S.: \emph{{D}ifferential {G}eometry, {L}ie {G}roups, and {S}ymmetric
  {S}paces}, Academic Press, New York, 1978.
  
\bibitem{Jo11}
Jost, J.: {\em Riemannian Geometry and Geometric Analysis}, Springer, Berlin, Heidelberg, New York, 2011.

\bibitem{LaSch15}
Lang, A. and Schwab, C.: {I}sotropic {G}aussian random fields on the sphere:
  regularity, fast simulation and stochastic partial differential equations,
  \emph{Ann. Appl. Probab.} \textbf{25} (2015)(6) 3047--3094.
  
\bibitem{Pe06}
Petersen, P.: {\em Riemannian Geometry}, Springer, New York, 2006.

\bibitem{Po09a}
Potthoff, J.: {S}ample properties of random fields{~II}: {C}ontinuity,
  \emph{Commun. Stoch. Anal.} \textbf{3} (2009) 331--348.

\bibitem{Sl37}
Slutsky, E.: {Q}ualche proposizione relative alla teoria delle funzioni
  aleatorie, \emph{Giorn. Ist. Attuari} \textbf{8} (1937) 183--199.

\end{thebibliography}
\end{document}